\documentclass[reqno,11pt]{amsart}

\usepackage{amsmath}
\usepackage{amsfonts}
\usepackage{amssymb}
\usepackage{color}
\usepackage{mathrsfs}
\usepackage[colorlinks=true]{hyperref}

\usepackage{cleveref}

\newtheorem{thm}{Theorem}[section]

\newtheorem{prop}[thm]{Proposition}

\newtheorem{rem}[thm]{Remark}

\numberwithin{equation}{section}

\newcommand{\aequation}{\renewcommand{\theequation}{\mbox{A.\arabic{equation}}}}

\newcommand{\nequation}{\setcounter{equation}{0}}

\newcommand{\R}{\mathbb{R}}
\newcommand{\C}{\mathbb{C}}

\newcommand{\Bb}{\mathbb{B}}

\newcommand{\Ac}{\mathcal{A}}
\newcommand{\B}{\mathcal{B}}

\newcommand{\ml}{\mathcal{L}}

\newcommand{\ve}{\varepsilon}
\newcommand{\rd}{\mathrm{d}}

\newcommand{\bear}{\begin{eqnarray}} 
\newcommand{\eear}{\end{eqnarray}} 
\newcommand{\bean}{\begin{eqnarray*}} 
\newcommand{\eean}{\end{eqnarray*}} 
\newcommand{\bs}{\begin{split}}
\newcommand{\es}{\end{split}}

\newcommand{\dhr}{\mathrel{\lhook\joinrel\relbar\kern-.8ex\joinrel\lhook\joinrel\rightarrow}}

\begin{document}

\title[]{On a Quasilinear Parabolic-Hyperbolic System Arising in MEMS Modeling}

\author{Christoph Walker}
\email{walker@ifam.uni-hannover.de}
\address{Leibniz Universit\"at Hannover\\ Institut f\" ur Angewandte Mathematik \\ Welfengarten 1 \\ D--30167 Hannover\\ Germany}
 
\date{\today}

\begin{abstract}
A coupled system consisting of a quasilinear  parabolic equation and a semilinear  hyperbolic  equation is considered. The problem arises  as a small aspect ratio limit in the modeling of a MEMS device taking into account the gap width of the device and the gas pressure. The system is regarded as a special case of a more general setting for which local well-posedness of strong solutions is shown. The general result applies to different cases including a coupling of the  parabolic equation to a semilinear wave equation of either second or fourth order, the latter featuring either clamped or pinned boundary conditions.
\end{abstract}


\maketitle

\section{Introduction}\label{S1}

Electrostatically actuated micro-electromechanical systems (MEMS) are ubiquitous in today's electronic devices. Idealized MEMS often consist of a fixed ground plate and an elastic membrane (or plate) that are close. Keeping the two components at different potential  induces a Coulomb force deflecting the membrane. In the past two decades MEMS devices have been a highly active  mathematical research focus, in particular due to their interesting qualitative behaviors with respect to pull-in instabilities (as a result from a possible touching of membrane and ground plate) and the inherent challenges related to local and global well-posedness of the corresponding models.  We refer to~\cite{LW_BAMS} and the references therein for more details on MEMS models and their mathematical investigation in general.

In this paper, we consider a model introduced in~\cite{GHL24,GHL23} arising as a small aspect ratio limit of equations governing an electrostatically actuated  MEMS, where the narrow gap separating the membrane and the ground plate is filled with a rarefied gas. More precisely, we consider
\begin{subequations}\label{PPPP} 
\begin{align}
\partial_t (wu) \, &=    \mathrm{div}\big(w^3u\nabla u\big) \,, \qquad t>0\, ,\quad x\in \Omega\, ,\label{P1}\\ 
\partial_t^2 w+\sigma \partial_t w&=\Delta w-\frac{a}{w^2}+b(u-1)\,, \qquad t>0\, ,\quad x\in \Omega\,,\label{P2} \\
u(t,x)&=\theta_1\,,\quad w(t,x)=\theta_2\,,\qquad t>0\,,\quad x\in\partial\Omega\,,\label{P3} \\
u(0,x)&=  u_0(x)\,, \quad w(0,x)=  w_0(x)\,,  \quad \partial_tw(0,x)=  w_0'(x)\,,\qquad x\in \Omega
\,,\label{P4}
\end{align}
\end{subequations}
where~$w=w(t,x)$ denotes  the varying width of the gap and~$u=u(t,x)$ is the local pressure of the gas. The (sufficiently) smooth bounded subset $\Omega$  of $\R^n$ with $n\in\{1,2\}$ represents the shape of the membrane and the ground plate.  
The constants $a,b,\theta_1,\theta_2>0$ and $\sigma\ge 0$ in~\eqref{P2}~-~\eqref{P3} as well as the initial data $u_0$, $w_0$, and $w_0'$ in~\eqref{P4} are given. The degeneracy of~\eqref{P1} and the singularity in~\eqref{P2} occurring for a vanishing gap width $w(t,x)=0$ capture the instabilities related to a touchdown of the membrane on the ground plate.  A detailed account of the modeling aspects is given in~\cite{GHL24,GHL23} to which we refer.

In~\cite{GHL24} the short-time existence of solutions to this  MEMS model
is established for the one-dimensional case $n=1$. The approach chosen therein consists of solving first the hyperbolic equation for $w$ (via a fixed point argument for a given, fixed $u$) and so reducing the coupled system~\eqref{PPPP} to a single fixed point equation for $u$ which is then solved using parabolic semigroup theory. Instead of decoupling the system, we proceed differently and solve the mild formulation of~\eqref{PPPP} simultaneously for $u$ and $w$, also relying on semigroup theory for semilinear hyperbolic and quasilinear parabolic equations described in \cite{A04,Pazy} respectively \cite{AmannTeubner,LQPP}. A key ingredient for this is the observation that mild solutions to the hyperbolic equation~\eqref{P2} enjoy {\it a priori} H\"older continuity properties with respect to time (and values in spaces of sufficiently high spatial regularity) that guarantee an evolution operator for the quasilinear parabolic equation~\eqref{P1} in the sense of~\cite{LQPP} (see also Remark~\ref{key} and Proposition~\ref{PropApp} below). In this way
we provide a considerably shorter proof for local existence including also the case $n=2$:

\begin{thm}\label{T}
Let $r>0$ and $u_0\in H^{2+r}(\Omega)$ with $u_0>0$ in $\Omega$ and $u_0=\theta_1$ on $\partial\Omega$. Let $w_0\in H^2(\Omega)\cap C^1(\bar\Omega)$, $w_0'\in H^1(\Omega)$ with $w_0>0$ in~$\Omega$ and $w_0-\theta_2=w_0'=0$  on~$\partial\Omega$. Then there is a unique solution
\begin{align*}
&u\in C^1\big([0,T],L_2(\Omega)\big)\cap C\big([0,T],H^2(\Omega)\big)\,,\\
&w\in C^2\big([0,T],L_2(\Omega)\big)\cap C^1\big([0,T],H^1(\Omega)\big)\cap C\big([0,T],H^2(\Omega)\big)
\end{align*}
to \eqref{PPPP} on some interval $[0,T]$.
\end{thm}

The initial values are compatible with the Dirichlet boundary conditions and the regularity of strong solutions at $t=0$. In fact, the solution component $u$ has even better regularity properties than stated in Theorem~\ref{T}, see Remark~\ref{R17}. Moreover, the solution can be extended to a maximal solution on $[0,T_{max})$ existing as long as $u(t)>0$ and $w(t)>0$ in $\bar \Omega$ as well as the $C^1$-norm of $(u,w)$ does not blow up. It is worth pointing out that a common feature in MEMS models~\cite{LW_BAMS} is the possible occurrence of  finite time quenching $\inf_{x\in \Omega}w(t,x) \to 0$ as $t\to T_{max}$  preventing global existence of solutions. \\

A similar result as Theorem~\ref{T} can also be shown for a related fourth-order equation when the Laplacian $\Delta$ in \eqref{P2} is replaced by $-\Delta^2+\Delta$  subject to pinned or clamped boundary conditions (see Theorem~\ref{T3} below for details). This corresponds to a MEMS device involving  an elastic plate instead of a membrane. 

In fact, Theorem~\ref{T} (and Theorem~\ref{T3}) can be regarded as a special case of a more general setting including a quasilinear parabolic equation coupled to a semilinear wave equation of the form
\begin{subequations}\label{G}
\begin{align}
\partial_t u&=\Ac(u,w)u +g(u,w,\partial_t w)\,,\quad t>0\,,\qquad u(0)=u_0\,,\\
\partial_t^2 w+\sigma \partial_t w&=-A w+f(u,w,\partial_t w)\,,\quad t>0\,,\qquad (w(0),\partial_tw(0))=(w_0,w_0')\,,
\end{align}
\end{subequations}
where $\Ac(u,w)$ are generators of analytic semigroups on a Banach space and $-A$ is a generator of a cosine function on a Hilbert space (see Section~\ref{Sec3} below for details).

In the next Section~\ref{Sec2} we first identify~\eqref{PPPP} as a special case of~\eqref{G} (see~\eqref{Uxy} below). The latter is treated in Section~\ref{Sec3} in an abstract functional analytic framework that is not restricted to the particular setting of~\eqref{PPPP}. The main result of this research is  Theorem~\ref{TT} on the local well-posedness of~\eqref{G} that is established using semigroup theory and  then implies Theorem~\ref{T} for the particular case~\eqref{PPPP} as shown in Section~\ref{Sec4}. In Section~\ref{Sec5} we briefly show how to apply Theorem~\ref{TT} for the case of the fourth-order problem~\eqref{XPPPP} including a bi-Laplacian.

\section{Functional Formulation of the Problem}\label{Sec2}

We demonstrate how to express the system~\eqref{PPPP} in the abstract form of problem~\eqref{G} and list relevant properties of the  functions involved.

Setting $$(\bar u,\bar w):=(u-\theta_1, w-\theta_2)\,,\qquad (\bar u_0,\bar w_0):=(u_0-\theta_1, w_0-\theta_2)$$ and dropping then again the bars for simplicity, problem~\eqref{PPPP} is equivalent to
\begin{subequations}\label{PPPP5} 
\begin{align}
\partial_t u \, &=   \frac{1}{w+\theta_2} \mathrm{div}\big((w+\theta_2)^3(u+\theta_1)\nabla u\big)-\frac{\partial_t w\,(u+\theta_1)}{w+\theta_2} \,, \qquad t>0\, ,\quad x\in \Omega\, ,\label{P15}\\ 
\partial_t^2 w+\sigma \partial_t w&= \Delta w-\frac{a}{(w+\theta_2)^2}+b(u+\theta_1-1)\,, \qquad t>0\, ,\quad x\in \Omega\,,\label{P25} \\
u(t,x)&= w(t,x)=0\,,\qquad t>0\,,\quad x\in\partial\Omega\,,
 \label{P35}\\
u(0,x)&=  u_0(x)\,, \quad w(0,x)=   w_0(x)\,,  \quad \partial_tw(0,x)=  w_0'(x)\,,\qquad x\in \Omega\,,\label{P45}
\end{align}
\end{subequations}
as long as $w>-\theta_2$. In the following, we shall focus on~\eqref{PPPP5} with initial data $u_0\in H^{2+r}(\Omega)$, $w_0\in H^2(\Omega)\cap C^1(\bar\Omega)$,  and $w_0'\in  H^1(\Omega)$ satisfying $u_0+\theta_1>0$ and $w_0+\theta_2>0$ in $\bar\Omega$ and 
$$
u_0=w_0=w_0'=0\quad \text{on }\ \partial\Omega\,.
$$ 

For technical reasons we handle the parabolic equation for $u$ in an $L_q$-setting, denoting by $H_q^s(\Omega)$ the scale of Bessel potential spaces \cite{AmannTeubner} (that coincide in the Hilbert space case $q=2$ with the Sobolev-Slobodeckii spaces $H^s(\Omega)=H_2^s(\Omega)$). 

Since $r>0$ and $n\in\{1,2\}$ we may choose $q\in (2,4)$ and $\alpha, \beta, \mu\in (0,1)$ such that
\begin{equation}\label{g5x}
\begin{split}
&\min\{r,1/2\}> n/2-n/q\,,\qquad\alpha\in \big(n/2-n/q,1/2\big)\,,\\ 
&\mu\in (2-2/q-\alpha,3/2-\alpha)\,,\qquad  2\beta\in (1,2)\,.
\end{split}
\end{equation}
Then, since $q>2\ge n$ and $\mu+\alpha>2-2/q\ge 1+n/2-n/q$, we have
\begin{align}\label{g5}
 H_q^{2\beta}(\Omega)\hookrightarrow H_q^{1}(\Omega)\hookrightarrow C(\bar\Omega)\,,\qquad   H^{1+\alpha}(\Omega)\hookrightarrow H^{\mu+\alpha}(\Omega)\hookrightarrow H_q^{1}(\Omega)\hookrightarrow C(\bar\Omega)
\end{align}
and 
$$
u_0\in H^{2+r}(\Omega)\hookrightarrow H_q^{2}(\Omega)\hookrightarrow C^1(\bar\Omega)\,,\qquad w_0\in  H^{2}(\Omega)\hookrightarrow H^{1+\alpha}(\Omega)
$$
with $u_0+\theta_1\ge 2\varsigma>0$ and $w_0+\theta_2\ge 2\varsigma>0$ in $\bar\Omega$ for some $\varsigma>0$. Hence we may choose $\ve>0$ such that
\begin{subequations}\label{g99}
\begin{align}\label{g90}
u+\theta_1\ge \varsigma\ \text{ in }\ \bar\Omega\,,\qquad u\in {\mathbb{B}}_{H_q^{2\beta}}(u_0,\ve) \,,\\
 w+\theta_2\ge \varsigma\ \text{ in }\ \bar\Omega\,,\qquad w\in {\mathbb{B}}_{H^{\mu+\alpha}}(w_0,\ve) \,.\label{g90b}
\end{align}
\end{subequations}
Set 
\begin{align}
\Ac(u,w)v:&= \frac{1}{w+\theta_2} \mathrm{div}\big((w+\theta_2)^3(u+\theta_1)\nabla v\big)\nonumber \\
&=  \mathrm{div}\big((w+\theta_2)^2(u+\theta_1)\nabla v\big)+ (w+\theta_2)(u+\theta_1)\nabla w\cdot\nabla v\label{AAq}
\end{align}
for $u, w\in C^1(\bar\Omega)$ and $v$ belonging to
$$
H_{q,D}^{2}(\Omega):=H_q^{2}(\Omega)\cap \mathring{H}_q^{1}(\Omega)
$$ 
(i.e. $ H_{q,D}^{2}(\Omega)$ incorporates homogeneous Dirichlet boundary conditions). Using the embeddings~\eqref{g5} along with the fact that $H_q^1(\Omega)$ is an algebra (since $q>n$) we obtain that the pointwise multiplications
$$
H^{\mu+\alpha}(\Omega)\bullet H^{\mu+\alpha}(\Omega)\bullet H_q^{2\beta}(\Omega)\bullet H_q^1(\Omega) \longrightarrow H_q^1(\Omega)
$$
and, choosing $\epsilon>0$ small with $\mu+\alpha-1-\epsilon> n/2-n/q$, 
$$
H^{\mu+\alpha}(\Omega)\bullet  H_q^{2\beta}(\Omega)\bullet H^{\mu+\alpha-1}(\Omega)\bullet H_q^1(\Omega) \longrightarrow H^{\mu+\alpha-1-\epsilon}(\Omega)\hookrightarrow L_q(\Omega)
$$
are continuous (see \cite[Theorem~4.1]{AmannMult} for the latter).
Consequently, using the first multiplication for the divergence term and the second multiplication for the first-order term of $\Ac(u,w)v$ in \eqref{AAq}, we derive that
\begin{subequations}\label{g66}
\begin{align}\label{g6b}
\Ac\in C^{1-}\big(H_q^{2\beta}(\Omega)\times H^{\mu+\alpha}(\Omega),\mathcal{L}(H_{q,D}^2(\Omega),L_q(\Omega))\big)\,,
\end{align}
where $C^{1-}$ means (locally) Lipschitz continuous. Moreover, since $u_0, w_0\in C^1(\bar\Omega)$ satisfy~\eqref{g99}, $\Ac(u_0,w_0)$ is a second-order normally elliptic differential operator in divergence form with $C^1$-coefficients. Hence, it follows from e.g. \cite[Theorem~4.1, Examples~4.3]{AmannTeubner} that
\begin{align}\label{g6}
\Ac(u_0,w_0)\in \mathcal{H}\big(H_{q,D}^2(\Omega),L_q(\Omega)\big)\,,
\end{align}
\end{subequations}
where $\mathcal{H}(E_1,E_0)$ denotes the set of  generators of analytic semigroups on the Banach space~$E_0$ with domain $E_1$.  Also observe the  identities
\begin{align}\label{int}
H_{q,D}^{2\theta}(\Omega):=\big[ L_q(\Omega), H_{q,D}^{2}(\Omega)\big]_\theta \doteq \left\{\begin{array}{lll}
\{v\in H_q^{2\theta}(\Omega)\,;\, v=0 \text{ on } \partial\Omega\}\, ,& 1/q <2\theta\le 2\,,\\[2mm]
H_q^{2\theta}(\Omega)\, ,& 0 \le 2\theta < 1/q\,,
\end{array}
\right.
\end{align}
for complex interpolation \cite[Theorem~5.2]{AmannTeubner}. Setting formally
\begin{align}\label{g}
g(u,w,w'):=-\frac{w'(u+\theta_1)}{w+\theta_2} \,, 
\end{align}
we may now reformulate~\eqref{P15} subject to the initial and boundary condition as quasilinear parabolic problem
\begin{align*}
\partial_t u=\Ac(u,w)u +g(u,w,\partial_t w)\,,\quad t>0\,,\qquad u(0)=u_0\,,
\end{align*}
in $L_q(\Omega)$.\\

We focus next on the hyperbolic problem for $w$. Let $H_{D}^{2\theta}(\Omega):=H_{2,D}^{2\theta}(\Omega)$ be as in~\eqref{int}. 
Clearly, the  Laplacian subject to homogeneous Dirichlet boundary conditions
\begin{align*}
-A:=\Delta\in \mathcal{H}\big(H_D^{2}(\Omega),L_2(\Omega)\big)
\end{align*}
is the generator of an analytic semigroup on $L_2(\Omega)$. In fact, 
\begin{equation}
\begin{split}\label{Ass}
&\text{$A: H_D^2(\Omega)\subset L_2(\Omega)\to L_2(\Omega)$ is  a closed, densely defined,}\\
&\text{self-adjoint, positive operator  with  compact inverse}\,.
\end{split}
\end{equation}
Introducing
\begin{equation}\label{f}
f(u,w):=-\frac{a}{(w+\theta_2)^2}+b(u+\theta_1-1)\,,
\end{equation}
we can rewrite~\eqref{PPPP5} now in the form
\begin{subequations}\label{Uxy}
\begin{align}
\partial_t u&=\Ac(u,w)u +g(u,w,\partial_t w)\,,\quad t>0\,,\qquad u(0)=u_0\,,\label{U1xs}\\
\partial_t^2 w+\sigma \partial_t w&=-A w+f(u,w)\,,\quad t>0\,,\qquad (w(0),\partial_tw(0))=(w_0,w_0')\,.\label{U2xs}
\end{align}
\end{subequations}
To handle the semilinear terms we define the open subsets
$$
 O_\beta:=\mathbb{B}_{H_{q,D}^{2\beta}}(u_0,\ve)\,,\qquad \mathbb{O}_\alpha:=\mathbb{B}_{H_{D}^{1+\alpha}}(w_0,\ve/c_0) \times H_D^{\alpha}(\Omega)
$$
of $ H_{q,D}^{2\beta}(\Omega)$ respectively $H_{D}^{1+\alpha}(\Omega)\times H_{D}^{\alpha}(\Omega)$, where $\ve>0$ stems from~\eqref{g99} and
$c_0>0$ denotes the norm of the embedding $H_{D}^{1+\alpha}(\Omega)\hookrightarrow H_{D}^{\mu+\alpha}(\Omega)$.
Recalling that $\alpha>n/2-n/q$ we find $2\eta\in (0,1/q)$ and $\epsilon>0$ small with $\alpha-\epsilon-n/2>2\eta-n/q$, hence the embedding 
$$
H^{\alpha-\epsilon}(\Omega)\hookrightarrow H_q^{2\eta}(\Omega)=H_{q,D}^{2\eta}(\Omega)\,.
$$ Therefore, noticing that
$$
(u+\theta_1)(w+\theta_2)^{-1}\in H_q^1(\Omega)\,,\qquad u\in O_\beta\,,\qquad (w,w')\in \mathbb{O}_\alpha\,,
$$
due to~\eqref{g5},~\eqref{g90b}, it follows from the continuity of the multiplication (see \cite[Theorem~4.1]{AmannMult})
$$
H_q^1(\Omega)\bullet  H^{\alpha}(\Omega)\longrightarrow H^{\alpha-\epsilon}(\Omega)\hookrightarrow H_{q,D}^{2\eta}(\Omega)
$$
and~\eqref{g} that
\begin{align}\label{gg1}
g\in C^{1-}\big(O_\beta\times \mathbb{O}_\alpha,H_{q,D}^{2\eta}(\Omega)\big)\,,
\end{align}
while~\eqref{g5},~\eqref{g90b},~\eqref{f}, and 
$$
H_q^1(\Omega)\hookrightarrow H^{\alpha}(\Omega)=H_{D}^{\alpha}(\Omega)
$$ 
(since $q>2$ and $\alpha<1/2$) ensure that (of course,~$f$ is independent of the $w'$-component)
\begin{align}\label{gg2}
 f\in C^{1-}\big(O_\beta\times \mathbb{O}_\alpha,H_{D}^{\alpha}(\Omega)\big)\,.
\end{align}
To guarantee later on sufficient regularity of solutions it is worth noting that
\begin{align}\label{fg}
&f(\hat u,\hat w)=-\frac{a}{(\hat w+\theta_2)^2}+b(\hat u+\theta_1-1)\in C^1\big([0,T],L_2(\Omega)\big)
\end{align}
whenever  $\hat u\in C^1\big([0,T],L_2(\Omega)\big)$ and $\hat w\in C\big([0,T],H^{1+\alpha}(\Omega)\big)\cap C^1\big([0,T],H^{\alpha}(\Omega)\big)$ with $\hat w(t)+\theta_2\ge \varsigma$ in $\Omega$.
Also note from $H^{2+r}(\Omega)\hookrightarrow H_q^2(\Omega)$ (due to \eqref{g5x}) that
\begin{align}\label{init}
&u_0 \in O_\beta\cap H_{q,D}^{2}(\Omega) \,,\qquad (w_0,w_0')\in  \mathbb{O}_\alpha\cap \big(H_{D}^{2}(\Omega)\times H_{D}^{1}(\Omega)\big) 
\end{align}
 by the assumptions of Theorem~\ref{T} and \eqref{g5}. In fact, since $u_0,w_0\in H^{2+r}(\Omega)$ and since $H^{1+r}(\Omega)$ is an algebra, it readily follows from~\eqref{AAq} that  
\begin{align}\label{init2}
u_0\in H_{q,D}^2(\Omega)\,,\qquad \Ac(u_0,w_0)u_0 \in H^{r}(\Omega)\hookrightarrow H_{q,D}^{2\eta}(\Omega)
\end{align}
since we may make $\eta>0$ smaller to guarantee $0<2\eta<\min\{r-n/2+n/q,1/q\}$. That is, the initial value $u_0$ belongs to the domain of the $H_{q,D}^{2\eta}(\Omega)$-realization of the generator $\Ac(u_0,w_0)\in\mathcal{H}\big(H_{q,D}^{2}(\Omega),L_q(\Omega)\big)$.\\

The previous considerations ensure that problem~\eqref{Uxy}  (and thus problem~\eqref{PPPP}) fits into the more general framework of Theorem~\ref{TT} of the  next section. We shall then continue from here in Section~\ref{Sec4} and finish off the proof of Theorem~\ref{T}  by applying~Theorem~\ref{TT}.

\section{Main Theorem}\label{Sec3}

As just pointed out above, Theorem~\ref{T} is a special case of a more general setting: Consider
\begin{subequations}\label{Ux}
\begin{align}
\partial_t u&=\Ac(u,w)u +g(u,w,\partial_t w)\,,\quad t>0\,,\qquad u(0)=u_0\,,\label{U1x}\\
\partial_t^2 w+\sigma \partial_t w&=-A w+f(u,w,\partial_t w)\,,\quad t>0\,,\qquad (w(0),\partial_tw(0))=(w_0,w_0')\,,\label{U2x}
\end{align}
\end{subequations}
where $\Ac(u,w)\in\mathcal{L}(E_1,E_0)$  for some continuously and densely injected Banach couple $E_1\hookrightarrow E_0$ is such that $\Ac(u_0,w_0)\in\mathcal{H}(E_1,E_0)$ (i.e. $\Ac(u_0,w_0)$ with domain $E_1$ generates an analytic semigroup on~$E_0$), and
\begin{equation}
\begin{split}\label{As}
&\text{$A: D(A)\subset H\to H$ is  a closed, densely defined, self-adjoint,}\\
&\text{positive operator  with bounded and compact inverse}
\end{split}
\end{equation}
on a Hilbert space $H$ with scalar product $(\cdot\vert\cdot)$. Here, positive operator means that $(Ax\vert x)\ge 0$ for $x\in D(A)$.  Let~\mbox{$\sigma\ge 0$}.\\

We formulate~\eqref{Ux} as a coupled system of two first order equations  relying on results for cosine functions \cite[Section 5.5 $\&$ Section 5.6]{A04}, see also Appendix~\ref{Appendix}. To this end note that \eqref{As} ensures that the powers $A^z$ for $z\in \C$ are well-defined closed operators (bounded for~$z\le 0$). Consequently, the matrix operator  
$$
\mathbb{A}:=\left(\begin{matrix} 0 & 1\\ -A & -\sigma\end{matrix}\right)\,,\quad D(\mathbb{A}):=D(A)\times D(A^{1/2})\,,
$$
generates a strongly continuous semigroup $(e^{t\mathbb{A}})_{t\ge 0}$ on the Hilbert space $$\mathbb{H}:=D(A^{1/2})\times H$$ (in fact, it generates a group $(e^{t\mathbb{A}})_{t\in \R}$). 
Using the notion ${\bf w}=(w,w')$ and setting
$$
{\bf F}(u,{\bf w}):=\left(\begin{matrix}
0 \\ f(u,w,w')
\end{matrix}\right)\,,
$$
we can write \eqref{U2x} as a semilinear hyperbolic Cauchy problem
$$
\partial_t {\bf w}=\mathbb{A}{\bf w} +{\bf F}(u,{\bf w})\,,\quad t>0\,,\quad {\bf w}(0)={\bf w}_0:=(w_0,w_0')\,,
$$
in $\mathbb{H}$. In fact, for greater flexibility (and to cope with the particular case~\eqref{PPPP5}) we shift this problem to the interpolation space (for some $\alpha\in [0,1)$) 
$$
\mathbb{H}_\alpha:=[\mathbb{H},D(\mathbb{A})]_\alpha \doteq D(A^{(1+\alpha)/2})\times D(A^{\alpha/2})\,,
$$
where we recall  (due to the Fourier series representation of $A^{\alpha}$ or \cite[Theorem~1.15.3]{Triebel}) that
\begin{align}\label{interpoll}
\big[D(A^{\alpha_0}),D(A^{\alpha_1})\big]_\theta \doteq D(A^{(1-\theta)\alpha_0+\theta\alpha_1}) \,,\qquad \theta\in [0,1]\,,\quad 0\le \alpha_0 <\alpha_1\,.
\end{align}
Then, the $\mathbb{H}_\alpha$-realization  $\mathbb{A}_\alpha$ of $\mathbb{A}$,
given by
$$
\mathbb{A}_\alpha {\bf w}:=\mathbb{A}{\bf w}\,,\quad {\bf w}\in D(\mathbb{A}_\alpha):=\{{\bf w}\in D(\mathbb{A})\,;\, \mathbb{A}{\bf w}\in \mathbb{H}_\alpha\}=D(A^{1+\alpha/2})\times D(A^{(1+\alpha)/2})\,,
$$  
generates a strongly continuous semigroup $(e^{t\mathbb{A}_\alpha})_{t\ge 0}$ on $\mathbb{H}_\alpha$ according to~\cite[Chapter~V]{LQPP}. 
We shall then consider~\eqref{Ux} in the  equivalent form
\begin{subequations}\label{Uxx}
\begin{align}
\partial_t u&=\Ac(u,w)u +g(u,w,\partial_t w)\,,\quad t>0\,,\qquad u(0)=u_0\,,\label{U1xx}\\
\partial_t {\bf w}&=\mathbb{A}_\alpha{\bf w} +{\bf F}(u,{\bf w})\,,\quad t>0\,,\qquad {\bf w}(0)={\bf w}_0=(w_0,w_0')\,.\label{U2xx}
\end{align}
\end{subequations}

In the following, let $(\cdot,\cdot)_\theta$ be arbitrary admissible interpolation functors~\cite[I.Section~2.11]{LQPP} and set
$$
E_\theta:=(E_0,E_1)_\theta\,,\qquad \theta\in [0,1]\,.
$$ 
Let $O_\beta\subset E_\beta$ and ${\bf \mathbb{O}}_\alpha\subset \mathbb{H}_\alpha$ be open sets for some $\alpha,\beta\in [0,1)$.

\begin{thm}\label{TT}
Let $\alpha,\beta,\mu\in [0,1)$  and $\tau\in (\beta,1]$. Consider initial values $u_0\in O_\beta\cap E_\tau$ and $(w_0,w_0')\in \mathbb{O}_\alpha$, let
\begin{equation}\label{i1}
\mathcal{A}\in C^{1-}\big(E_\beta\times D(A^{(\alpha+\mu)/2}),\mathcal{L}(E_1,E_0)\big)\,,\qquad \Ac(u_0,w_0)\in\mathcal{H}(E_1,E_0)\,,
\end{equation}
and suppose~\eqref{As}. Moreover, assume that
\begin{equation}\label{i2}
g\in C^{1-}(O_\beta\times \mathbb{O}_\alpha, E_0)\,,\qquad f\in C^{1-}\big(O_\beta\times \mathbb{O}_\alpha, D(A^{\alpha/2})\big)\,.
\end{equation}

{\bf (a)} There is a unique mild solution
\begin{equation}\label{g0}
u\in C\big([0,T],E_\beta\big)\,,\qquad {\bf w}=(w,\partial_t w)\in C\big([0,T],D(A^{(1+\alpha)/2})\times D(A^{\alpha/2})\big)
\end{equation}
to the Cauchy problem~\eqref{Uxx} on some interval $[0,T]$. \\

{\bf (b)}  If 
\begin{equation}\label{g1}
g\in C(O_\beta\times \mathbb{O}_\alpha,  E_\eta)\ \text{ for some $\eta>0$}
\end{equation}
or if 
\begin{equation}\label{g2}
g\in C^{1-}\big(O_\beta\times D(A^{(\alpha+\mu)/2}),E_0\big)\ \text{ is independent of~$w'$}\,,
\end{equation}
 then $u\in C^1\big((0,T],E_0\big)\cap C\big((0,T],E_1\big)$ is a strong solution to~\eqref{U1x}. \\

{\bf (c)}  Let $u_0\in E_1$. If 
\begin{equation}\label{g11}
\text{\eqref{g1} is satisfied and }\ \Ac(u_0,w_0)u_0\in E_\eta
\end{equation}
or if \eqref{g2} is satisfied,
then 
\begin{equation}\label{uu1}
u\in C^1\big([0,T],E_0\big)\cap C\big([0,T],E_1\big)
\end{equation} is a strict solution to~\eqref{U1x}. In this case,
if 
\begin{equation}
\begin{split}\label{g33}
&f(u,w,\partial_t w)\in C\big([0,T], D(A^{1/2})\big)
\end{split}
\end{equation}
or if   
\begin{equation}
\begin{split}\label{g44}
&f\, \text{ is independent of~$w'$ and }\ f(u,w)\in C^1\big([0,T], H\big) \,,
\end{split}
\end{equation}
then 
$$
w \in C^2\big([0,T], H\big)\cap C^1\big([0,T], D(A^{1/2})\big)\cap C\big([0,T], D(A)\big)
$$
is a strong solution to~\eqref{U2x} provided that $(w_0,w_0')\in \mathbb{O}_\alpha\cap D(A)\times D(A^{1/2})$.
\end{thm}

We emphasize that one may rely on the regularity properties~\eqref{g0} and~\eqref{uu1} when checking~\eqref{g33} or~\eqref{g44}.

\begin{proof} 
{\bf (i)} It follows from~\eqref{i1} and \cite[I.Theorem~1.3.1]{LQPP} that there are $\ve>0$, $\kappa\ge 1$,  and $\omega>0$ such that\footnote{The notation $\Ac\in\mathcal{H}\big(E_1,E_0;\kappa,\omega\big)$ means that $\omega-\Ac\in\ml is(E_1,E_0)$ and $$\kappa^{-1}\le\frac{\|(\lambda-\Ac)x\|_{E_0}}{\vert\lambda\vert\,\|x\|_{E_0}+\| x\|_{E_1}}\le \kappa\,,\qquad x\in E_1\setminus\{0\}\,,\quad \mathrm{Re}\,\lambda\ge \omega\,, $$ see \cite[I.Section~1.2]{LQPP}. Note that $\mathcal{H}\big(E_1,E_0;\kappa,\omega\big)\subset \mathcal{H}(E_1,E_0)$.}
\begin{align}\label{g6s}
\Ac(u,w)\in \mathcal{H}\big(E_1,E_0;\kappa,\omega\big)\,,\quad (u,w)\in\bar\Bb_{E_\beta\times D(A^{(\alpha+\mu)/2})}\big((u_0,w_0),\ve\big)\,,
\end{align}
and
\begin{align}\label{g6bs}
\|\Ac(u,w)-\Ac(\hat u,\hat w)\|_{\mathcal{L}(E_1,E_0)}\le c\, \|(u, w)-(\hat u,\hat w)\|_{E_\beta\times D(A^{(\alpha+\mu)/2})}\,,\nonumber\\ (u,w), (\hat u,\hat w)\in\bar\Bb_{E_\beta\times D(A^{(\alpha+\mu)/2})}\big((u_0,w_0),\ve\big)\,, 
\end{align}
for some constant $c=c(u_0,w_0)>0$. 
Let $\rho\in (0,\min\{\tau-\beta,1-\mu\})$ and let
$$
c_*:=\max\big\{1,\|i\|_{\ml(D(A^{(\alpha+1)/2}),D(A^{(\alpha+\mu)/2}))}\big\}\,,
$$
where $i:D(A^{(\alpha+1)/2})\hookrightarrow D(A^{(\alpha+\mu)/2})$ is the natural inclusion. Writing $z=(u,{\bf w})$ with ${\bf w}=(w,w')$ in the following and noticing that $z_0=(u_0,{\bf w}_0)\in O_\beta\times \mathbb{O}_\alpha$ with open subsets $O_\beta \subset E_\beta$ and $\mathbb{O}_\alpha\subset \mathbb{H}_\alpha$, it follows from \eqref{i2} that we may assume without loss of generality that
\begin{align}\label{Lipsch}
\|g(z)-g(\hat z)\|_{E_0}+\|{\bf F}(z)-{\bf F}(\hat z)\|_{\mathbb{H}_\alpha}&\le c_1\|z-\hat z\|_{E_\beta\times \mathbb{H}_\alpha}\,,\quad z,\hat z\in \bar\Bb_{E_\beta\times \mathbb{H}_\alpha}\left(z_0,\ve/c_*\right)\,,
\end{align}
for some constant $c_1=c_1(z_0)>0$ and  that
$$
\|(u,w)-(u_0,w_0)\|_{E_\beta\times D(A^{(\alpha+\mu)/2})}\le 
\ve\,,\qquad z\in \bar\Bb_{E_\beta\times \mathbb{H}_\alpha}\big(z_0,\ve/c_*\big)\,.
$$
Given $T\in (0,1)$ (to be specified later) we then introduce
\begin{align*}
\mathcal{V}_T:=\Big\{z=(u,&{\bf w})\in C\big([0,T],E_\beta\times \mathbb{H}_\alpha\big)\,;\,w\in C^1\big([0,T], D(A^{\alpha/2})\big)\,,\\
& \|z(t)-z_0\|_{E_\beta\times \mathbb{H}_\alpha}\le \ve/c_*\,,\,\ \|u(t)-u(s)\|_{E_\beta}\le \vert t-s\vert^\rho\,,\\
& \|\partial_t w(t)\|_{D(A^{\alpha/2})}\le \|{\bf w}_0\|_{\mathbb{H}_\alpha}+\ve/2c_*\,,\, t,s\in [0,T] \Big\}\,,
\end{align*}
where $z_0=(u_0,{\bf w}_0)$ and the notation ${\bf w}=(w,w')$ is used throughout. Then $\mathcal{V}_T$ is a complete metric space when equipped with the metric 
$$
d_{\mathcal{V}_T}(z,\hat z):=\|z-\hat z\|_{C([0,T], E_\beta\times \mathbb{H}_\alpha)}+ \|\partial_t w-\partial_t \hat w\|_{C([0,T],D(A^{\alpha/2}))}\,.
$$
Then, for $z=(u,{\bf w})\in \mathcal{V}_T$, we have by interpolation (see~\eqref{interpoll})
\begin{equation}\label{hold}
w\in C\big([0,T],D(A^{(1+\alpha)/2})\big)\cap C^1\big([0,T],D(A^{\alpha/2})\big)\hookrightarrow C^{1-\mu}\big([0,T],D(A^{(\alpha+\mu)/2})\big)
\end{equation}
and it thus follows from \eqref{g6bs} and $\rho<1-\mu$ that
\begin{subequations}\label{LP}
\begin{equation}\label{g8s}
\begin{split}
&\sup_{0\le s<t\le T}\frac{\|\mathcal{A}(u(t),w(t))-\mathcal{A}(u(s),w(s))\|_{\mathcal{L}(E_1,E_0)}}{\vert t-s\vert^\rho}\le r(u_0,{\bf w}_0)
\end{split}
\end{equation}
for some constant $r(u_0,{\bf w}_0)>0$ (independent of $z\in  \mathcal{V}_T$) and from  \eqref{g6s} that
\begin{align} \label{reso}
\mathcal{A}\big(u(t),w(t)\big)\in \mathcal{H}\big(E_1,E_0;\kappa,\omega\big) \,,\quad t\in [0,T]\,.
\end{align}
\end{subequations}
Now, \cite[II.Corollary~4.4.2]{LQPP} and \eqref{LP} imply that for each $z=(u,{\bf w})\in \mathcal{V}_T$, the operator $\mathcal{A}(u,w)$ generates a parabolic evolution operator
$$
U_{\mathcal{A}(u,w)}(t,s)\,,\quad 0\le s\le t\le T\,,
$$
on $E_0$ with regularity subspace $E_1$ and that we may apply the results of \cite[II.Section~5]{LQPP}.
Introduce now 
\begin{subequations}
\begin{align} 
\Gamma_1(z)(t)&:=U_{\mathcal{A}(u,w)}(t,0)u_0+\int_0^t U_{\mathcal{A}(u,w)}(t,s)\, g(z(s))\,\rd s\,,\label{Gamma1s}\\
{\bf \Gamma}_2(z)(t)&:=e^{t\mathbb{A}_\alpha}{\bf w}_0+\int_0^t e^{(t-s)\mathbb{A}_\alpha}\, {\bf F}(z(s))\,\rd s\,,\label{Gamma2s}
\end{align}
\end{subequations}
for  $t\in [0,T]$ and $z=(u,{\bf w})\in \mathcal{V}_T$ recalling $u_0\in O_\beta \cap E_\tau$ and ${\bf w}_0=(w_0,w_0')\in  \mathbb{H}_\alpha$. Then, mild solutions to~\eqref{Uxx} correspond to fixed points of the operator $\Gamma=(\Gamma_1,{\bf \Gamma}_2)$. \\

{\bf (ii)}  We claim that  $\Gamma: \mathcal{V}_T\to \mathcal{V}_T$ is a contraction for $T\in (0,1)$ sufficiently small. To see this, let $z=(u,{\bf w})\in \mathcal{V}_T$. It then follows from \eqref{Gamma1s}, \eqref{Lipsch}, \eqref{LP}, and \cite[II.Theorem~5.3.1]{LQPP} that (for some $c>0$ depending only on the parameters in \eqref{LP})
\begin{align*}
\| \Gamma_1(z)(t)-\Gamma_1(z)(s)\|_{E_\beta}\le c\,(t-s)^{\tau-\beta}\, \big(\|u_0\|_{E_\tau}+\|g(z)\|_{ C([0,T],E_0)})\le \vert t-s\vert^\rho
\end{align*}
(we recall that $\rho<\tau-\beta$) and thus, in particular,
\begin{align*}
\| \Gamma_1(z)(t)-u_0\|_{E_\beta}\le T^\rho\le \frac{\ve}{2 c_*}
\end{align*}
for $0\le s\le t\le T$ with $T\in (0,1)$ sufficiently small. Moreover,  we deduce  from \eqref{Lipsch} that ${\bf F}(z)\in C([0,T],\mathbb{H}_\alpha)$ so that \eqref{Gamma2s}, the assumption ${\bf w}_0\in \mathbb{H}_\alpha$, and Proposition~\ref{PropApp}  entail that 
$$
{\bf \Gamma}_2(z)=(\Gamma_2(z),\Gamma_2(z)')\in C([0,T],\mathbb{H}_\alpha)$$ 
with 
\begin{equation}\label{ghs}
\Gamma_2(z)\in C^1\big([0,T],D(A^{\alpha/2})\big)\,, \quad\partial_t\Gamma_2(z)=\Gamma_2(z)'\,, 
\end{equation}
and
\begin{align*}
\| {\bf \Gamma}_2(z)(t)-{\bf w}_0\|_{\mathbb{H}_\alpha}&\le  \|e^{t\mathbb{A}_\alpha}{\bf w}_0-{\bf w}_0\|_{\mathbb{H}_\alpha}+\int_0^t \|e^{(t-s)\mathbb{A}_\alpha}\|_{\ml(\mathbb{H}_\alpha)}\, \|{\bf F}(z(s))\|_{\mathbb{H}_\alpha}\,\rd s\le  \frac{\ve}{2 c_*}
\end{align*}
for $t\in [0,T]$ with $T\in (0,1)$ sufficiently small since $(e^{t\mathbb{A}_\alpha})_{t\ge 0}$ is strongly continuous on~$\mathbb{H}_\alpha$. 
In particular, \eqref{ghs} implies
$$
\|\partial_t\Gamma_2(z)(t)\|_{D(A^{\alpha/2})}\le \| {\bf \Gamma}_2(z)(t)\|_{\mathbb{H}_\alpha}\le  \|{\bf w}_0\|_{\mathbb{H}_\alpha}+ \frac{\ve}{2 c_*}\,,\quad t\in [0,T]\,.
$$
Consequently, $\Gamma: \mathcal{V}_T\to \mathcal{V}_T$ is well-defined.\\

{\bf (iii)}  To check the Lipschitz property consider $z=(u,{\bf w})\in \mathcal{V}_T$ and $\hat z=(\hat u,{\bf \hat w})\in \mathcal{V}_T$. Then \eqref{LP} and \cite[II.Theorem~5.2.1]{LQPP} imply
\begin{align*}
\| \Gamma_1&(z)(t)-\Gamma_1(\hat z)(t)\|_{E_\beta}\\
&\le c\,\Big\{t^{\tau-\beta}\,\|\mathcal{A}(u,w)-\mathcal{A}(\hat u,\hat w)\|_{C([0,T],\ml(E_1,E_0))}\,
\big[\,\|u_0\|_{E_\tau}+t^{1-\tau}\|g(z)\|_{ C([0,T],E_0)}\,\big]\\
&\qquad\qquad +t^{1-\beta}\|g(z)-g(\hat z)\|_{ C([0,T],E_0)}
\Big\}\\
&\le \frac{1}{4}\, d_{\mathcal{V}_T}\big(z,\hat z\big)
\end{align*}
for $t\in [0,T]$ with $T\in (0,1)$ sufficiently small, where we used \eqref{g6bs} and \eqref{Lipsch} for the last estimate.
Moreover,  due to \eqref{Lipsch} we have
\begin{align*}
\| &{\bf \Gamma}_2(z)(t)-{\bf \Gamma}_2(\hat z)(t)\|_{\mathbb{H}_\alpha}\le \int_0^t \|e^{(t-s)\mathbb{A}_\alpha}\|_{\ml(\mathbb{H}_\alpha)}\, \|{\bf F}(z(s))-{\bf F}(\hat z(s))\|_{\mathbb{H}_\alpha}\,\rd s\le \frac{1}{8}\, d_{\mathcal{V}_T}\big(z,\hat z\big)
\end{align*}
for $t\in [0,T]$ with $T\in (0,1)$ sufficiently small. 
Consequently, taking into account~\eqref{ghs}, we deduce that
$$
d_{\mathcal{V}_T}\big(\Gamma(z),\Gamma(\hat z)\big)\le \frac{1}{2}\, d_{\mathcal{V}_T}\big(z,\hat z\big)\,,\quad z, \hat z\in \mathcal{V}_T\,.
$$
 Banach's fixed point theorem now ensures that there is a unique $z=(u,{\bf w})\in \mathcal{V}_T$  with $z=\Gamma(z)$; that is, $(u,{\bf w})$ is a mild solution to~\eqref{Ux}. 

This proves statement~{\bf (a)} of Theorem~\ref{TT}. \\

{\bf (iv)} Setting for $t\in [0,T]$
\begin{subequations}\label{AA}
\begin{equation}
\tilde\Ac(t):=\Ac(u(t),w(t))\,,\qquad \tilde g(t):=g\big(u(t),w(t),\partial_t w(t)\big)\,,
\end{equation}
we see that $u$ is a mild solution to the linear Cauchy problem
\begin{equation}\label{AA21b}
\partial_t u=\tilde\Ac(t)u +\tilde g(t)\,,\quad t\in (0,T]\,,\qquad u(0)=u_0\,.
\end{equation}
\end{subequations}
If \eqref{g1} holds, then  $\tilde g\in C([0,T], E_\eta)$ with $\eta>0$ and we infer from~\cite[II.Theorem~1.2.2]{LQPP} that $$u\in  C^1((0,T],E_0)\cap C((0,T],E_1)$$ is a strong solution to~\eqref{U1x}. If \eqref{g2} holds, then \eqref{Lipsch} and~\eqref{hold} imply $\tilde g\in C^\rho([0,T], E_0)$ and we obtain again that $u$ is a strong solution to~\eqref{U1x} with regularity properties as above in view of~\cite[II.Theorem~1.2.1]{LQPP}.

This proves statement~{\bf (b)} of Theorem~\ref{TT}. \\

{\bf (v)}  Let now $u_0\in E_1$. Then \eqref{g2} or \eqref{g11}  both imply~\eqref{uu1}  due to~\cite[II.Theorem~1.2.1]{LQPP} respectively~\cite[II.Theorem~1.2.2]{LQPP}; that is,
$$
u\in C^1\big([0,T],E_0\big)\cap C\big([0,T],E_1\big)
$$
is a strict solution to~\eqref{U1x}. 
Set now $\tilde{{\bf F}}(t):={\bf F}(z(t))$ and note that ${\bf w}$ is a mild solution to
the linear Cauchy problem
$$
\partial_t {\bf w}=\mathbb{A}{\bf w} +\tilde{{\bf F}}(t)\,,\quad t\in (0,T]\,,\qquad {\bf w}(0)={\bf w}_0\,.
$$
Then \eqref{g33} implies $\tilde{{\bf F}}\in C([0,T],D(\mathbb{A}))$ while~\eqref{g44}
yields \mbox{$\tilde{{\bf F}}\in C^1([0,T],\mathbb{H})$}. In either case we derive from
 Proposition~\ref{PropApp} that
$$
w \in C^2\big([0,T], H\big)\cap C^1\big([0,T], D(A^{1/2})\big)\cap C\big([0,T], D(A)\big)
$$
is a strong solution to~\eqref{U2x} provided that $(w_0,w_0')\in D(\mathbb{A})$. This proves statement~{\bf (c)} of Theorem~\ref{TT}.
\end{proof}

\begin{rem}\label{key}
It is worth emphasizing that one of the key ingredients of the proof of Theorem~\ref{TT} is the observation that the first component $w$ of a \emph{mild} solution ${\bf w}=(w,w')$ to the hyperbolic equation~\eqref{U2xx} enjoys a H\"older continuity property with respect to time and values in spaces of sufficiently high spatial regularity (see~\eqref{hold}) as stated in Proposition~\ref{PropApp}. In fact, this ensures the H\"older continuity of the operator $t\mapsto \mathcal{A}(u(t),w(t))$ and thus that the associated evolution operator is well-defined according to~\cite[II.Corollary~4.4.2]{LQPP}.  
\end{rem}

\section{Proof of Theorem~\ref{T}}\label{Sec4}

We can now complete the proof of Theorem~\ref{T}. From Section~\ref{Sec2} we know that problem~\eqref{PPPP} is equivalent to \eqref{PPPP5} (recalling that $(u,w)$ is identified with $(u-\theta_1,w-\theta_2)$) which, in turn, is a special case of~\eqref{Ux} (see~\eqref{Uxy}).

Choose $q\in (2,4)$ and $\alpha, \beta, \mu\in (0,1)$ as in \eqref{g5x} and $\eta\in (0,1)$ as in \eqref{gg1} and \eqref{init2}. Setting 
$$
E_0:=L_q(\Omega)\,,\qquad E_1:=H_{q,D}^2(\Omega)\,,
$$
we notice from~\eqref{int} that $E_\theta\doteq H_{q,D}^{2\theta}(\Omega)$ (with complex interpolation functor) for $2\theta\not=1/q$ while the operator $A:= -\Delta$ in $H:=L_2(\Omega)$ with domain $H_D^{2}(\Omega)$ satisfies~\eqref{As} (see~\eqref{Ass}). Moreover,~\eqref{int} and \eqref{interpoll} imply
\begin{equation}\label{domains}
D(A^s)\doteq H_D^{2s}(\Omega)\,,\qquad s\in [0,1]\,,\quad 2s\not= 1/2\,.
\end{equation}
It now follows from~\eqref{g66}~-~\eqref{init2} that problem~\eqref{Uxy} (and thus problem~\eqref{PPPP}) fits into the framework of Theorem~\ref{TT} with assumptions~\eqref{i1}, \eqref{i2}, \eqref{g1}, \eqref{g11}, and \eqref{g44}  satisfied (and $\tau=1$). Therefore, Theorem~\ref{TT} implies that \eqref{PPPP5} admits a unique solution
\begin{align*}
&u\in C^1\big([0,T],L_q(\Omega)\big)\cap C\big([0,T],H_{q}^2(\Omega)\big)\,,\\
&w\in C^2\big([0,T],L_2(\Omega)\big)\cap C^1\big([0,T],H^1(\Omega)\big)\cap C\big([0,T],H^2(\Omega)\big)\,.
\end{align*}
Since $q>2$, this proves Theorem~\ref{T}. \qed

\begin{rem}\label{R17}
As shown above, $u$ belongs in fact to $C^1\big([0,T],L_q(\Omega)\big)\cap C\big([0,T],H_{q}^2(\Omega)\big)$ for some $q>2$. Parabolic smoothing effects ensure additional regularity properties. For instance, the regularity of $(u,w)$ stated in Theorem~\ref{T} implies that $u$ solves a (linear) equation of the form~\eqref{AA} with $\tilde{\mathcal{A}}\in C^\rho([0,T],\mathcal{H}(H_D^2(\Omega),L_2(\Omega)))$ (see~\eqref{LP}) and $\tilde{g}\in C^\rho([0,T],L_2(\Omega))$ for some $\rho>0$ (see~\eqref{g}). The maximal regularity result of~\cite[I.Theorem~1.2.1]{LQPP} yields $u\in C^\rho((0,T], H^2(\Omega))\cap  C^{1+\rho}((0,T], L_2(\Omega))$. Moreover, since $\tilde{g}\in C([0,T],H_{q,D}^{2\eta}(\Omega))$ by~\eqref{gg1}, a higher spatial regularity of $u$ is derived from~\cite[I.Theorem~1.2.2]{LQPP} taking into account the regularity~\eqref{init2} of the initial value.
\end{rem}

\section{A Fourth-Order Problem}\label{Sec5}

As pointed out in the introduction, Theorem~\ref{TT} also applies to certain fourth-order wave equations. Indeed, consider
\begin{subequations}\label{XPPPP} 
\begin{align}
\partial_t (wu) \, &=    \mathrm{div}\big(w^3u\nabla u\big) \,, \qquad t>0\, ,\quad x\in \Omega\, ,\label{XP1}\\ 
\partial_t^2 w+\sigma \partial_t w&=-D_1\Delta^2w+D_2\Delta w-\frac{a}{w^2}+b(u-1)\,, \qquad t>0\, ,\quad x\in \Omega\,,\label{XP2} \\
u(t,x)&=\theta_1\,,\quad w(t,x)=\theta_2\,,\quad \mathcal{B} w(t,x)=0\,,\qquad t>0\,,\quad x\in\partial\Omega\,,\label{XP3} \\
u(0,x)&=  u_0(x)\,, \quad w(0,x)=  w_0(x)\,,  \quad \partial_tw(0,x)=  w_0'(x)\,,\qquad x\in \Omega
\,,\label{XP4}
\end{align}
with 
\begin{align}\label{XP5}
\mathcal{B}w:=(1-\delta)\partial_\nu w+\delta \Delta w\ \text{ on }\ \partial\Omega\,,
\end{align}
\end{subequations}
where $\delta\in\{0,1\}$.  Equations~\eqref{XPPPP} govern the gap width $w$ and the gas pressure $u$ for a MEMS device involving an elastic plate (instead of a membrane) of shape  $\Omega$, where $\Omega\subset\R^n$ with $n\in\{1,2\}$ is assumed to be a (sufficiently) smooth bounded set. The elastic plate is either clamped at its boundary (corresponding to $\delta=0$) or is hinged along its boundary so that it is free to rotate (corresponding to $\delta=1$).
We assume that $D_1>0$ and $D_2\ge 0$ and that $a,b,\theta_1,\theta_2>0$ and $\sigma\ge 0$. 

This  MEMS model was introduced in~\cite{GHL23} and the short-time existence of solutions was established for the pinned case $\delta=1$ (for both cases $n=1,2$). 
We derive the result for pinned and clamped boundary conditions simultaneously as a consequence of Theorem~\ref{TT} (the assumptions on the initial values are compatible with the regularity of  the solution):

\begin{thm}\label{T3}
Let $r>0$ and $u_0\in H^{2+r}(\Omega)$ with $u_0>0$ in $\Omega$ and $u_0=\theta_1$ on $\partial\Omega$. Let $(w_0,w_0')\in H^4(\Omega)\times H^2(\Omega)$ with $w_0>0$ in~$\Omega$ and $w_0-\theta_2=\B w_0=0$  on $\partial\Omega$ and $w_0'=(1-\delta)\B w_0'=0$  on $\partial\Omega$. Then there is a unique solution
\begin{align*}
&u\in C^1\big([0,T],L_2(\Omega)\big)\cap C\big([0,T],H^2(\Omega)\big)\,,\\
&w\in C^2\big([0,T],L_2(\Omega)\big)\cap C^1\big([0,T],H^2(\Omega)\big)\cap C\big([0,T],H^4(\Omega)\big)
\end{align*}
to \eqref{XPPPP} on some interval $[0,T]$.
\end{thm}

\begin{proof} The proof is very much the same as for Theorem~\ref{T} and we thus only sketch it and point out  the new aspects.
Arguing as in Section~\ref{Sec2} by shifting $u$ and $w$ we may  focus on
\begin{subequations}\label{XUx}
\begin{align}
\partial_t u&=\Ac(u,w)u +g(u,w,\partial_t w)\,,\quad t>0\,,\qquad u(0)=u_0\,,\label{XU1x}\\
\partial_t^2 w+\sigma \partial_t w&=-A w+f(u,w)\,,\quad t>0\,,\qquad (w(0),\partial_tw(0))=(w_0,w_0')\,,\label{XU2x}
\end{align}
\end{subequations}
with $\Ac$ defined in~\eqref{AAq} and $g$ and $f$ in \eqref{g} respectively~\eqref{f}. The only difference now is that we  consider the fourth-order operator 
\begin{align*}
-A:=-D_1\Delta^2+D_2\Delta\in \mathcal{L}\big(H_\B^{4}(\Omega),L_2(\Omega)\big)\,,
\end{align*}
where its domain $H_\B^{4}(\Omega)$ incorporates the homogeneous pinned ($\delta=0$) or clamped ($\delta=1$) boundary conditions. More generally, we set for $s\in [0,4]$
$$
H_\B^{s}(\Omega):=\left\{\begin{array}{lll}
& \big\{v\in H^{s}(\Omega)\,;\, v=\B v=0 \text{ on } \partial\Omega\big\}\, , & s>3/2+\delta\ ,\\[2mm]
& \big\{v\in H^{s}(\Omega)\,;\, v=0 \text{ on } \partial\Omega\big\}\, , & 1/2<s<3/2+\delta\ ,\\[2mm]
&   H^{s}(\Omega)\, , & 0\le s<1/2\ .
\end{array}
\right.
$$
Then $H_{\B}^{4\theta}(\Omega)$  coincides with the complex interpolation space	
\begin{align}\label{Xinterpol}
H_\B^{4\theta}(\Omega) \doteq \big[ L_2(\Omega), H_\B^{4}(\Omega)\big]_\theta\ ,\quad 4\theta\in [0,4]\setminus\left\{1/2 , 3/2+\delta\right\}\ ,
\end{align}
up to equivalent norms, see \cite[Theorem 4.3.3]{Triebel}.
Moreover, $-A\in \mathcal{H}\big(H_\B^{4}(\Omega),L_2(\Omega)\big)$ is the generator of an analytic semigroup on $L_2(\Omega)$ with exponential decay, e.g., see \cite[Remarks~4.2]{AmannTeubner} or \cite[Theorem 7.2.7]{Pazy}. In fact, we have again that 
\begin{equation}
\begin{split}\label{XAss}
&\text{$A: H_\B^{4}(\Omega)\subset L_2(\Omega)\to L_2(\Omega)$ is  a closed, densely defined,}\\
&\text{self-adjoint, positive operator  with  compact inverse}\,.
\end{split}
\end{equation}
Moreover, \eqref{Xinterpol} and \eqref{interpoll} entail
\begin{equation}\label{domainss}
D(A^s)\doteq H_\B^{4s}(\Omega)\,,\qquad s\in [0,1]\,,\quad 4s\not= 1/2\,,\, 3/2+\delta\,.
\end{equation}
As in Section~\ref{Sec2} we choose $q>2$ and $\alpha, \beta, \mu\in (0,1)$ such that
\begin{align*}
r\ge n/2-n/q\,,\quad 2\alpha\in \big(n/2-n/q,1/2\big)\,,\quad \mu\in (1-\alpha,1)\,, \quad 2\beta\in (1,2)\,,
\end{align*}
and define the open subsets 
$$
 O_\beta:=\mathbb{B}_{H_{q,D}^{2\beta}}(u_0,\ve)\,,\qquad \mathbb{O}_\alpha:=\mathbb{B}_{H_{\B}^{2+2\alpha}}(w_0,\ve/c_0) \times H_{\B}^{2\alpha}(\Omega)
$$
of $ H_{q,D}^{2\beta}(\Omega)$ respectively $H_{\B}^{2+2\alpha}(\Omega)\times H_{\B}^{2\alpha}(\Omega)$, where
$c_0>0$ denotes the norm of the embedding $H_{\B}^{2+2\alpha}(\Omega)\hookrightarrow H_{\B}^{2\mu+2\alpha}(\Omega)$ and $\varepsilon>0$ is such that, for some $\varsigma>0$,
\begin{align*}
u+\theta_1\ge \varsigma\ \text{ in }\ \bar\Omega\,,\qquad u\in O_\beta=\mathbb{B}_{H_{q,D}^{2\beta}}(u_0,\ve) \,,\\
 w+\theta_2\ge \varsigma\ \text{ in }\ \bar\Omega\,,\qquad w\in {\mathbb{B}}_{H^{2\mu+2\alpha}}(w_0,\ve) \,. 
\end{align*}
Exactly as in Section~\ref{Sec2} we have
\begin{align}\label{Xg6b}
\Ac\in C^{1-}\big(H_q^{2\beta}(\Omega)\times H^{2\mu+2\alpha}(\Omega),\mathcal{L}(H_{q,D}^2(\Omega),L_q(\Omega))\big)
\end{align}
and
\begin{align}\label{Xg6}
\Ac(u_0,w_0)\in \mathcal{H}\big(H_{q,D}^2(\Omega),L_q(\Omega)\big)\,,
\end{align}
while
\begin{align}\label{Xgg1}
g\in C^{1-}\big(O_\beta\times \mathbb{O}_\alpha,H_{q,D}^{2\eta}(\Omega)\big)\,,\qquad
 f\in C^{1-}\big(O_\beta\times \mathbb{O}_\alpha,H_{\B}^{2\alpha}(\Omega)\big)
\end{align}
for some $\eta>0$ small enough.
Moreover,
\begin{align}\label{Xfg}
&f(\hat u,\hat w)=-\frac{a}{(\hat w+\theta_2)^2}+b(\hat u+\theta_1-1)\in C^1\big([0,T],L_2(\Omega)\big)
\end{align}
whenever  $\hat u\in C^1\big([0,T],L_2(\Omega)\big)$ and $\hat w\in C\big([0,T],H^{2+2\alpha}(\Omega)\big)\cap C^1\big([0,T],H^{2\alpha}(\Omega)\big)$ with $\hat w(t)+\theta_2\ge \varsigma$ in $\Omega$.
Finally, by premise of the theorem,
\begin{align}\label{Xinit}
&u_0 \in O_\beta\cap H_{q,D}^{2}(\Omega) \,,\qquad (w_0,w_0')\in  \mathbb{O}_\alpha\cap \big(H_{\B}^{4}(\Omega)\times H_{\B}^{2}(\Omega)\big) 
\end{align}
and, as before, since $u_0,w_0\in H^{2+r}(\Omega)$,
\begin{align}\label{Xinit2}
\Ac(u_0,w_0)u_0 \in H_q^{2\eta}(\Omega)=H_{q,D}^{2\eta}(\Omega)
\end{align}
 making $\eta>0$ smaller, if necessary, so that $0<2\eta<\min\{r,1/q\}$. Setting 
$$
E_0:=L_q(\Omega)\,,\quad E_1:=H_{q,D}^2(\Omega)\,,\quad H:=L_2(\Omega)\,,
$$
it then follows from~\eqref{XAss}~-~\eqref{Xinit2} that problem~\eqref{XUx} (and thus problem~\eqref{XPPPP}) fits into the framework of Theorem~\ref{TT} with assumptions~\eqref{i1}, \eqref{i2}, \eqref{g1}, \eqref{g11}, and \eqref{g44}  satisfied (and $\tau=1$). Therefore, Theorem~\ref{TT} implies Theorem~\ref{T3}.
\end{proof}

\begin{appendix}

\nequation
\aequation

\section{}\label{Appendix}

Let $H$  be a Hilbert space and $A: D(A)\subset H\to H$ be  a closed, densely defined, self-adjoint, positive operator  with bounded and compact inverse, where its domain  $D(A)$ is equipped with the graph norm.
Then the square root $A^{1/2}$ (more generally: $f(A)$ for $f:(0,\infty)\to \C$) is a well-defined closed operator on $H$ by Fourier series representation.

\begin{prop}\label{PropApp}
Suppose~\eqref{As} and let $\sigma\in\R$.
The matrix operator 
\begin{equation*}
\mathbb{A}:= \begin{pmatrix} 0&1\\ -A&-\sigma \end{pmatrix}\ , \quad  D(\mathbb{A}):= D(A) \times
D(A^{1/2})\,,
\end{equation*}
generates a strongly continuous group on the Hilbert space $\mathbb{H}:=D(A^{1/2})\times H$ (with an exponential decay if $\sigma>0$). Consider ${\bf w_0}\in \mathbb{H}$ and 
$$ 
{\bf F}:=\left(\begin{matrix} 0\\ f\end{matrix}\right)\,,\quad  f\in C([0,T],H)\,.
$$
Then
$$
{\bf w}(t)=e^{t\mathbb{A}}{\bf w_0}+\int_0^t e^{(t-s)\mathbb{A}}\, {\bf F}(s)\,\rd s\,, \quad t\in [0,T]\,,
$$
satisfies ${\bf w}=(w,w')\in C([0,T],\mathbb{H})$ and
$$
w\in C^1\big([0,T],H\big) \ \text{ with }\ \partial_t w=w'\,.
$$
If ${\bf w_0}\in D(\mathbb{A})$ and ${\bf F}\in C^1([0,T],\mathbb{H})+C([0,T],D(\mathbb{A}))$, then $${\bf w}\in C^1([0,T],\mathbb{H})\cap C([0,T],D(\mathbb{A}))$$
is a strong solution to
$$
\partial_t{\bf w}=\mathbb{A}{\bf w}+{\bf F}(t)\,,\quad t \in [0,T]\,,\qquad {\bf w}(0)={\bf w}_0\,.
$$
\end{prop}

\begin{proof}
Let $\sigma=0$. Then 
$$
e^{t\mathbb{A}}=\begin{pmatrix} \cos\big(tA^{1/2}\big)&\sin\big(tA^{1/2}\big) A^{-1/2}\\-\sin\big(tA^{1/2}\big) A^{1/2}&\cos\big(tA^{1/2}\big)
\end{pmatrix}\ ,\quad t\in \R\,,
$$
and the mild formulation for ${\bf w}=(w,w')$ yields explicit formulas for both $w$ and $w'$ which readily imply that ${\bf w}=(w,w')\in C([0,T],\mathbb{H})$ with
$$
w\in C^1\big([0,T],H\big) \ \text{ with }\ \partial_t w=w'\,.
$$
If $\sigma\not=0$, then replace $f$ by $f-\sigma w'\in C([0,T],H)$ to reduce the problem to the case $\sigma=0$. The statement about  strong solutions is classical.
\end{proof}

\end{appendix}

\bibliographystyle{siam}
\bibliography{Lit_MEMS}

\begin{thebibliography}{1}

\bibitem{AmannMult}
{\sc H.~Amann}, {\em Multiplication in {S}obolev and {B}esov spaces}, in
  Nonlinear analysis, Sc. Norm. Super. di Pisa Quaderni, Scuola Norm. Sup.,
  Pisa, 1991, pp.~27--50.

\bibitem{AmannTeubner}
\leavevmode\vrule height 2pt depth -1.6pt width 23pt, {\em {Nonhomogeneous
  linear and quasilinear elliptic and parabolic boundary value problems}}, in
  {Function spaces, differential operators and nonlinear analysis
  ({F}riedrichroda, 1992)}, vol.~133 of {Teubner-Texte Math.}, Teubner,
  Stuttgart, 1993, p.~9–126.

\bibitem{LQPP}
\leavevmode\vrule height 2pt depth -1.6pt width 23pt, {\em Linear and
  quasilinear parabolic problems. {V}ol. {I}}, vol.~89 of Monographs in
  Mathematics, Birkh\"{a}user Boston, Inc., Boston, MA, 1995.
\newblock Abstract linear theory.

\bibitem{A04}
{\sc W.~Arendt}, {\em Semigroups and evolution equations: functional calculus,
  regularity and kernel estimates}, in Evolutionary equations. {V}ol. {I},
  Handb. Differ. Equ., North-Holland, Amsterdam, 2004, pp.~1--85.

\bibitem{GHL24}
{\sc H.~Gimperlein, R.~He, and A.~A. Lacey}, {\em Existence and uniqueness for
  a coupled parabolic-hyperbolic model of {MEMS}}, Math. Methods Appl. Sci., 47
  (2024), pp.~6310--6353.

\bibitem{GHL23}
\leavevmode\vrule height 2pt depth -1.6pt width 23pt, {\em Wellposedness of a
  nonlinear parabolic-dispersive coupled system modelling {MEMS}}, J.
  Differential Equations, 384 (2024), pp.~193--251.

\bibitem{LW_BAMS}
{\sc {\relax Ph}.~Lauren\c{c}ot and {\relax Ch}.~Walker}, {\em Some singular
  equations modeling {MEMS}}, Bull. Amer. Math. Soc. (N.S.), 54 (2017),
  pp.~437--479.

\bibitem{Pazy}
{\sc A.~Pazy}, {\em Semigroups of linear operators and applications to partial
  differential equations}, vol.~44 of Applied Mathematical Sciences,
  Springer-Verlag, New York, 1983.

\bibitem{Triebel}
{\sc H.~Triebel}, {\em {Interpolation Theory, Function Spaces, Differential
  Operators}}, North-Holland, Amsterdam, 1978.

\end{thebibliography}

\section*{Declarations}
\subsection*{Conflict of interest} The author declares that he has no conflict of interest.\\


\end{document}